\documentclass{conm-p-l}

\usepackage{amssymb}

\usepackage{graphicx}

\usepackage{}

\newcommand{\Z}{{\mathbb Z}}
\newcommand{\KK}{{\mathbb K}}
\newcommand{\PP}{{\mathbb P}}

\newcommand{\A}{{\mathcal A}}
\newcommand{\C}{{\mathbb C}}
\newcommand{\Sym}{\mathop{\rm Sym}\nolimits}
\newcommand{\ot}{\mathop{\rm OT}\nolimits}
\newcommand{\ao}{\mathop{\rm AOT}\nolimits}

\newcommand{\codim}{\mathop{\rm codim}\nolimits}

\newcommand{\IN}{\mathrm{in}}

\newcommand{\HS}{\mathrm{HS}}

\newtheorem{theorem}{Theorem}[section]
\newtheorem{proposition}[theorem]{Proposition}
\newtheorem{lemma}[theorem]{Lemma}

\theoremstyle{definition}
\newtheorem{definition}[theorem]{Definition}
\newtheorem{example}[theorem]{Example}

\theoremstyle{remark}

\numberwithin{equation}{section}

\begin{document}

\title[Subarrangements of type A]{Subarrangements of type A: the weak Lefschetz \\property of the Artinian Orlik-Terao algebra}


\author{Nicholas Gaubatz}
\address{Gaubatz: Department of Mathematics \& Statistics, Auburn University, Auburn, AL 36849}
\email{nmg0029@auburn.edu}

\author{Hal Schenck}
\address{Schenck: Department of Mathematics \& Statistics, Auburn University, Auburn, AL 36849}
\email{hks0015@auburn.edu}


\date{}

\begin{abstract}
    In \cite{OT1}, Orlik and Terao introduce a commutative Artinian analog $Sym(V^*)/I_{\A}$ of the Orlik-Solomon algebra of a hyperplane arrangement $\A$ to answer a question of Aomoto \cite{AO}. A central topic of investigation in the study of Artinian algebras is the Weak Lefschetz Property (WLP). We analyze WLP for the Artinian Orlik-Terao algebra of graphic arrangements. Even for chordal graphs (which give rise to Koszul algebras) WLP sometimes fails; conversely an analysis of the state polytope shows WLP can hold even when WLP fails for all possible initial ideals. More generally, for any algebra with a tensor product decomposition, we construct canonical elements in the kernel of the multiplication map, refining previous results in the literature.
\end{abstract}

\maketitle


\section{Introduction} \label{sec:one}

Let $\A=\{H_1,\dots ,H_d\}$ be an arrangement of complex hyperplanes in $\C^n$, with $H_i=V(\ell_i)$ and $\ell_i \in \C[x_1,\ldots,x_n]$. In their 1980 paper \cite{OS}, Orlik and Solomon prove that the cohomology ring of the 
complement $X=\C^n\setminus \bigcup_{i=1}^d H_i$ is determined by the intersection lattice $L(\A)$ (the poset of intersections, graded by codimension), showing that $H^*(X,\Z)$ is the quotient of the exterior algebra $E=\bigwedge (\Z^d)$ on generators
$e_1, \dots , e_d$ in degree one by the ideal generated by all those elements of the form $\partial e_{i_1\dots i_r}=\sum_{q}(-1)^{q-1}e_{i_1} \cdots
\widehat{e_{i_q}}\cdots e_{i_r}$ for which $\codim(H_{i_1}\cap \cdots \cap H_{i_r}) < r$.

\subsection{The Orlik-Terao algebra}

Orlik and Terao \cite{OT1} define a commutative version of $H^*(X,\Z)$ over a field $\mathbb{K}$ to answer a question of Aomoto \cite{AO}. 

\begin{definition}\label{OTdef}
    Let $[d]$ denote $\{1,\ldots ,d\}$, $H_i = V(\ell_i)$ and let $\Lambda =\{i_1,\ldots ,i_k\} \subset [d]$. If $\codim(\bigcap_{j=1}^k H_{i_j}) < k$, then there are constants $c_{t}$ such that
    \[
        \sum\limits_{t \in \Lambda} c_t \ell_t =0.
    \]
    For each dependency $\Lambda=\{i_1,\ldots ,i_k\}$, let $r_\Lambda = \sum_{j=1}^k c_{i_j}y_{i_j} \in S
    =\mathbb{K}[y_1,\ldots,y_d]$. Define 
    \[
        f_\Lambda = \partial(r_\Lambda) = \sum_{j=1}^k c_{i_j} \cdot y_{i_1}\cdots \hat
        y_{i_{j}} \cdots y_{i_k},
    \]
    and let $I$ be the ideal generated by the $f_{\Lambda}$. The {\em Orlik-Terao algebra} $\ot$ is the quotient of $\mathbb{K}[y_1,\ldots,y_d]$ by $I$, and is also known as the {\em algebra of reciprocal forms}, because the ideal $I$ also arises as the kernel of the map $Sym(\mathbb{K}^d) \rightarrow S$ given by $y_i \mapsto \frac{1}{\ell_i}.$ The {\em Artinian Orlik-Terao algebra} $\ao$ is the quotient of $\ot$ by the squares of the variables. In this note, we focus on the case of $char(\mathbb{K})=0$.
\end{definition}
Orlik-Terao show that the Hilbert series of AOT agrees with that of $H^*(X,\mathbb{K})$ so we have $\HS(\text{AOT}, t) = P(\mathcal{A}, t)$, the Poincar\'e polynomial (see \cite{OT1}, Theorem 4.3). Proudfoot-Speyer prove in \cite{PS} that it is Cohen-Macaulay and has a flat degeneration to the Stanley-Reisner ring of the broken circuit complex. A number of recent works investigate the geometry and matroid theoretic properties of OT and AOT, see for example \cite{DGT, FSW, GST, LM, PS, S, ST,T}. In this paper, we focus on the Artinian Orlik-Terao algebra for subarrangements of the reflecting hyperplanes of a type A reflection arrangement, which are also known as {\em graphic arrangements}.

\subsection{Graphic arrangements}

\begin{definition}
Let $G$ be a simple, connected graph on $\kappa_0$ vertices, with edge-set $\mathsf{E}$ with $|\mathsf{E}|=\kappa_1$. Then 
\[\mathcal{A}_G=\bigcup \{ V(x_i-x_j)\mid (i,j)\in \mathsf{E} \}
\]
is the corresponding arrangement in $\C^{\kappa_0}$.
\end{definition}

In \cite{K85}, Kohno connected an algebraic property of the cohomology ring (Koszulness) to a property of the intersection lattice $L(\A)$ of the arrangement:

\begin{definition}
    A pair $(x,y)$ of elements of a lattice $L$ is modular if for all $z \le y$, $z \vee (x \wedge y) = (z \vee x) \wedge y.$ An element $x$ is {\em modular} if it forms a modular pair for all $y \in L$, and $L$ is {\em supersolvable} if it has a maximal chain of modular elements.
\end{definition}

In \cite{STAN} Stanley shows that $L(\A_G)$ is supersolvable iff $G$ is chordal. For the Orlik-Solomon algebra, \cite{Y} proves that supersolvability is synonymous with having (in some term order) a quadratic Gr\"obner basis, and \cite{DGT} shows the same for the Orlik-Terao algebra. An algebra admitting (in some term order) a quadratic Gr\"obner basis is Koszul; for cohomology rings of graphic arrangements \cite{SS} shows they are equivalent; Question 4.7 of \cite{DGT} asks about this in the AOT setting. \begin{definition} A graph is {\em chordal} if every cycle with at least four vertices has a {\em chord}: an edge connecting two non-adjacent vertices in the cycle. \end{definition}
Let $\kappa_i$ denote the clique numbers of $G$: $\kappa_0 = \sharp \mbox{ vertices of G}, \kappa_1 = \sharp \mbox{ edges of G}$, and so on. As noted in Remark 6.18 of \cite{SS}, letting $\Delta_j = \sum_{s=j}^{\kappa_0-1}(-1)^{s-j}{s \choose j} \kappa_{s}$, the chromatic polynomial of a chordal graph $G_{ch}$ is given by
\[
    \chi_{G_{ch}}(t)=t^{\kappa_0} \prod_{j=1}^{\kappa_0-1}
    \left(1-jt^{-1}\right)^{\Delta_j},
\]

Using that $\chi_{G}(t) = t^{\kappa_0} P(\A_G, -t^{-1})$ where $P(\A.t) = \sum \dim H^i(X,\mathbb{K})t^i$ is the Hilbert series of the cohomology ring yields the Hilbert series of the AOT algebra of a chordal graph $G_{ch}$:
\begin{equation}
    \label{eq:chordHS}
    \HS(AOT(G_{ch}),t) = \prod_{j=1}^{\kappa_0-1} \left(1+jt\right)^{\Delta_j}.
\end{equation}


\subsection{The Weak Lefschetz Property}

Lefschetz properties appear in numerous areas of mathematics: algebra, combinatorics, geometry and topology. Lefschetz properties come in two flavors: Weak and Strong; in a remark at the close of their paper Proudfoot-Speyer note that after quotienting OT by a system of parameters (that is, the ``usual'' Artinian reduction, rather than by squares of the variables as in AOT) that the resulting Artinian algebra does not satisfy the Strong Lefschetz Property. However, AOT is a different algebra than the Artinian reduction of OT, and our focus is on the Weak (rather than Strong)  Lefschetz Property: 

\begin{definition}
    Let $I =\langle f_1, \ldots, f_k \rangle \subseteq S=\KK[x_1,\ldots, x_d]$ be an ideal with $A=S/I$ Artinian. 
    Then $A$ has the {\em Weak Lefschetz Property} $($WLP$)$ if there is 
    an $\ell \in S_1$ such that for all $i$, the multiplication map 
    $\mu_{\ell}: A_i \longrightarrow A_{i+1}$ has maximal rank. If not, we say that $A$ fails WLP in degree $i$. 
\end{definition}

The set of elements $\ell\in S_1$ with the property that the 
multiplication map $\mu_\ell$ has maximum rank is a (possibly empty) 
Zariski open set in $S_1$. Therefore, the existence of the Lefschetz 
element $\ell$ in the definition of the Weak Lefschetz Property 
guarantees that this set is  nonempty and is equivalent to the 
statement that for a general linear form in $S_1$ the 
corresponding multiplication map has full rank. 

The Weak Lefschetz Property depends on char($\KK$). For 
quadratic monomial ideals \cite{MNS} shows there is a natural connection to topology, and homology with $\mathbb{Z}/2$ coefficients plays a key role in understanding WLP.

Results of \cite{HMNW} show that WLP always holds for $d\le 2$ in characteristic zero.
For $d = 3$ WLP sometimes fails, but is known to hold for many classes: ideals of general forms \cite{A}, complete intersections \cite{HMNW}, ideals with semistable syzygy bundle \cite{BK}, almost complete intersections with unstable syzygy bundle \cite{BK}, ideals generated by powers of linear forms \cite{SS, MMN2}, monomial ideals of small type \cite{BMMRNZ, CN}. Nevertheless, even the $d = 3$ case is  not completely understood, and there remain intriguing open questions: for example, for $d=3$, does every Artinian Gorenstein algebra have WLP? For $d \ge 4$ far less is known; it is open if every $A$ which is a complete intersection has WLP. The survey paper \cite{MN} is a good introduction to the field; for more on Lefschetz properties, see the book \cite{HMMNWW}.

\subsection{Main Results}
Two main themes of the paper are understanding how WLP relates to graph theoretic constructions, and how such constructions relate to tensor decompositions of AOT algebras. In particular, we \begin{itemize}
        \item show there exist AOT algebras with WLP, but where no initial ideal possesses WLP; this is of interest as one method to show that $S/I$ has WLP is to show it for some initial ideal $\IN(I)$.
        \item construct canonical kernel elements of the multiplication map when the algebra has a tensor product decomposition.
    \item prove that forests, cycle graphs, and graphs with certain tensor product decompositions (and dimensions) have WLP, building on results of \cite{BMMRNZ}.
    \item use the Macaulay2 computer algebra system to categorize WLP for the AOT algebras of the graphic arrangements corresponding to all simple connected graphs on at most 7 vertices.
    \item discuss a way to build 5-vertex graphs having WLP from the 5-cycle, and the hurdles in extending this to 6-vertex graphs.
    
    \item describe open questions and directions for further research.
\end{itemize}


\section{Motivating Examples}

Our interest in studying this situation comes from the following data, obtained using the {\tt Macaulay2} computer algebra system and {\tt NautyGraphs} package. Denote a simple, connected graph as SC, and SCC if it is also chordal:

\begin{table}[h]
    \centering
    \begin{tabular}{|c|c|c|c|c|}
        \hline $\stackrel {\sharp}{\mbox{vertices}}$ & $\stackrel{\sharp \mbox{ iso. classes}}{\mbox{of SC graphs}}$ & $\stackrel{\sharp \mbox{ iso. classes}}{\mbox{of SCC graphs}}$ & $\stackrel{\sharp \mbox{ SC graphs}}{\mbox{failing WLP}}$ &$\stackrel{\sharp \mbox{ SCC graphs }}{\mbox{failing WLP}}$\\
        \hline $4$ & $6$       & $5$     & $0$& 0\\
        \hline $5$ & $21$   & $15$  & $1$ &1\\
        \hline $6$ & $112$  & $58$ & $6$ &3\\
        \hline $7$ & $853$  & $272$ & $79$ & $51$ \\
        \hline
    \end{tabular}
    \vskip .1in
    \caption{WLP results for graphs with at most seven vertices  }
    \vskip -.3in
\end{table}

\subsection{Building a stable of examples}

For simple connected graphs on four or fewer vertices, every corresponding AOT algebra has WLP. For five vertices, there is only one simple graph having an AOT failing WLP:

\begin{example}\label{bowtie}
    Let $G$ consist of two triangles, joined at a vertex. 
    \begin{figure}[h]
        \vskip -.12in
        \includegraphics[width=2in]{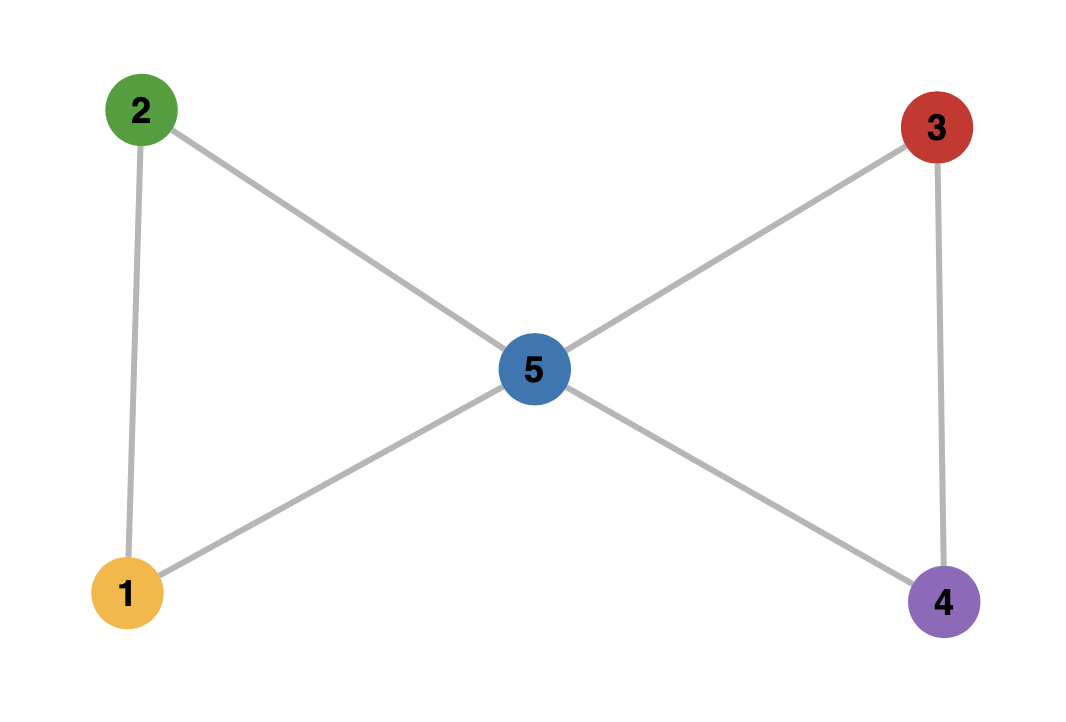} 
        \vskip -.2in
        \caption{The sole SCC graph on five vertices failing WLP}
    \end{figure}
    
    \noindent $G$ has six edges, and the relations from Definition~\ref{OTdef} are
    \[y_1y_2-y_1y_3+y_2y_3 \mbox{ and } y_4y_5-y_4y_6+y_5y_6\]
    From Equation~\ref{eq:chordHS} the Hilbert series for the AOT algebra is
    \[
        \HS(AOT,t) = 1+6t+13t^2+12t^3+4t^4.
    \]
    For this example, both injectivity and surjectivity fail: a computation shows the map $\mu_{\ell} = \cdot \sum a_iy_i: A_2 \rightarrow A_3$ has rank $11$. Notice that $A$ decomposes as a tensor product algebra $A'\otimes A^{''}$, where 
    \[
        \begin{array}{c}
            A' = \KK[y_1,y_2,y_3]/\langle y_1^2,y_2^2,y_3^2,y_1y_2-y_1y_3+y_2y_3\rangle \\
            \\
            A^{''}=\KK[y_4,y_5,y_6]/\langle y_4^2,y_5^2,y_6^2,y_4y_5-y_4y_6+y_5y_6 \rangle.
        \end{array}
    \]
    The Hilbert series for both is $1+3t+2t^2$; for $S/\IN(I)$ with a lexicographic term ordering where $y_{1} \succ \dots \succ y_{6}$, we have the decomposition:
    \[
        S/\IN(I) = \KK[y_1,y_2]/\langle y_1^2,y_1y_2,y_2^2\rangle \otimes \KK[y_3]/y_3^2 \otimes 
        \KK[y_4,y_5]/\langle y_4^2,y_4y_5,y_5^2\rangle\otimes \KK[y_6]/y_6^2.
    \]
    By Proposition~4.7 of \cite{MNS}, $S/\IN(I)$ fails WLP, but this need not imply that $A$ fails WLP.  With respect to bases for $A_2$ and $A_3$ as below
    \[
        \begin{array}{ccc}
            A_2&:& \{y_{1}y_{3},y_{1}y_{4},y_{1}y_{5},y_{1}y_{6},y_{2}y_{3},y_{2}y_{4},y_{2}y_{5},y_{2}y_{6},y_{3}y_{4},y_{3}y_{5},y_{3}y_{6},y_{4}y_{6},y_{5}y_{6}\}\\
            A_3&: &
            \{y_{1}y_{3}y_{4},y_{1}y_{3}y_{5},y_{1}y_{3}y_{6},y_{1}y_{4}y_{6},y_{1}y_{5}y_{6},y_{2}y_{3}y_{4},\\
            &&y_{2}y_{3}y_{5},y_{2}y_{3}y_{6},y_{2}y_{4}y_{6},y_{2}y_{5}y_{6},y_{3}y_{4}y_{6},y_{3}y_{5}y_{6}\}
        \end{array}
    \]
    \vskip .05in
    \noindent We find that $\mu_{\ell}$ is 
    \begin{small}
        \[
            \left[\!\begin{array}{ccccccccccccc}
               \vphantom{\left\{3,\:0\right\}}a_{4}&a_{2}\!+\!a_{3}&0&0&0&a_{1}&0&0&a_{1}&0&0&0&0\\
               \vphantom{\left\{3,\:0\right\}}a_{5}&0&a_{2}\!+\!a_{3}&0&0&0&a_{1}&0&0&a_{1}&0&0&0\\
               \vphantom{\left\{3,\:0\right\}}a_{6}&0&0&a_{2}\!+\!a_{3}&0&0&0&a_{1}&0&0&a_{1}&0&0\\
               \vphantom{\left\{3,\:0\right\}}0&a_{5}\!+\!a_{6}&a_{4}&a_{4}&0&0&0&0&0&0&0&a_{1}&0\\
               \vphantom{\left\{3,\:0\right\}}0&-a_{5}&a_{6}\!-\!a_{4}&a_{5}&0&0&0&0&0&0&0&0&a_{1}\\
               \vphantom{\left\{3,\:0\right\}}0&-a_{2}&0&0&a_{4}&a_{3}\!-\!a_{1}&0&0&a_{2}&0&0&0&0\\
               \vphantom{\left\{3,\:0\right\}}0&0&-a_{2}&0&a_{5}&0&a_{3}\!-\!a_{1}&0&0&a_{2}&0&0&0\\
               \vphantom{\left\{3,\:0\right\}}0&0&0&-a_{2}&a_{6}&0&0&a_{3}\!-\!a_{1}&0&0&a_{2}&0&0\\
               \vphantom{\left\{3,\:0\right\}}0&0&0&0&0&a_{5}\!+\!a_{6}&a_{4}&a_{4}&0&0&0&a_{2}&0\\
               \vphantom{\left\{3,\:0\right\}}0&0&0&0&0&-a_{5}&a_{6}\!-\!a_{4}&a_{5}&0&0&0&0&a_{2}\\
               \vphantom{\left\{3,\:0\right\}}0&0&0&0&0&0&0&0&a_{5}\!+\!a_{6}&a_{4}&a_{4}&a_{3}&0\\
               \vphantom{\left\{3,\:0\right\}}0&0&0&0&0&0&0&0&-a_{5}&a_{6}\!-\!a_{4}&a_{5}&0&a_{3}
               \end{array}\!\right]
        \]
    \end{small}
    
    \noindent  For the algebra $A'$ above, the kernel of the map $A'_1 \rightarrow A'_2$ is generated by 
    \[
        \begin{array}{c}
              \left(a_{1}+a_{2}-a_{3}\right)a_1y_{1}+\left(-a_{1}-a_{2}-a_{3}\right)a_2y_{2}+\left(-a_{1}+a_{2}+a_{3}\right)a_3y_{3}.
        \end{array}
    \]
    Changing subscripts $i\rightarrow i+3$, we obtain the same expression for the generator of the kernel of the map $A^{''}_1 \rightarrow A^{''}_2$. Lemma~\ref{taut} shows that tensoring these two gives an element of the kernel $A_2 \rightarrow A_3$: a vector of 13 bidegree $(2,2)$ quartics in the two sets of $a_i$ variables. The second element of $ker(\mu_{\ell})$ is identified in Propostion~\ref{taut2}.
\end{example}

A main tool for showing that WLP holds for $A=S/I$ is showing that WLP holds for $A=S/\IN(I)$ for some initial ideal of $I$, and we discuss this in \S 3. The next example shows one drawback of this strategy.

\begin{example}\label{BadNews}
    Let $G$ be obtained from the bipartite graph $K_{2,3}$ on vertex sets $\{v_1,v_2\}$ and $\{w_1,w_2,w_3\}$ by adding the edge $\overline{v_1v_2}$. 
    \begin{figure}[h]
        \vskip -.12in
        \includegraphics[width=1.7in]{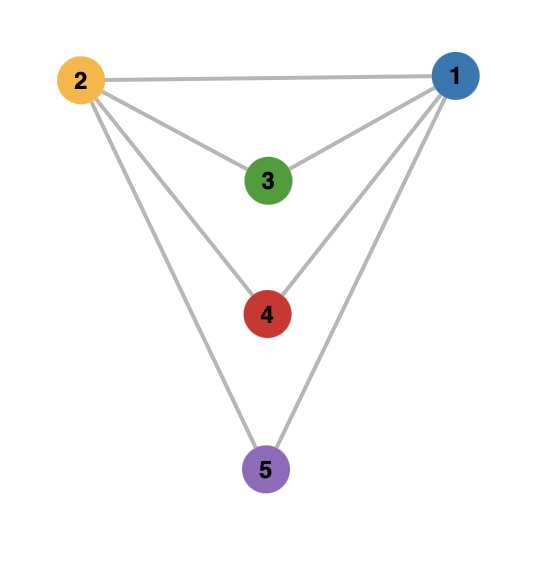} 
        \vskip -.24in
        \caption{Different term orders yield different WLP properties for $\IN(I)$}
    \end{figure}
    \vskip .05in
    \noindent For two different orders $\prec, \prec'$, suppressing the $y_i^2$ terms for clarity we have:
    \[
        \begin{array}{ccc}
            \IN_{\prec}(I) & =& \!\!\!\!\!\!\!\!\!\!\! \!\!\!\!\!\!\!\!\!\!\!\!\!\!\!\!\!\!\!\!\!\! \!\!\!\!\!\!\!\!\!\!\!\!\!\!\!\!\!\!\!\!\!\! \!\!\!\langle y_{1}y_{2},y_{3}y_{4}, y_{5}y_{6}\rangle\\
            \IN_{\prec'}(I) & =&\langle y_{1}y_{2},y_{1}y_{3}, y_{1}y_{4},y_{3}y_{4}y_{6},y_{2}y_{4}y_{5},y_{2}y_{3}y_{5}\rangle
        \end{array}
    \]
    The algebra $S/\IN_{\prec'}(I)$ has WLP, and $S/\IN_{\prec}(I)$ does not; the simplicial complex for $\IN_{\prec}(I)$ is the cone over the boundary of an octahedron. Because $G$ is chordal there exists a quadratic Gr\"obner basis, as displayed for $\prec$. While $\prec'$ does not yield a quadratic Gr\"obner basis, it does yield $\IN(I)$ having WLP. This is interesting: often a quadratic Gr\"obner basis reflects ``nice'' properties of $I$. 
\end{example}

\begin{example}\label{Bad6verts}
    For graphs on six vertices, there are six simple connected graphs which fail WLP; if we allow graphs to be disconnected there are two additional examples which we include for the interested reader.

    \begin{figure}[h]
        \includegraphics[width=3.7in]{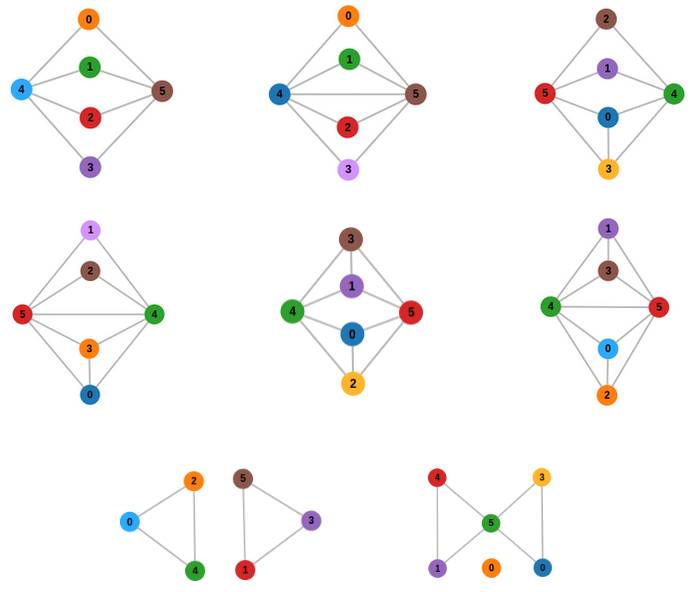} 
        \caption{The 6 isomorphism classes of simple, connected graphs on 6 vertices whose graphic arrangement AOT algebras fail WLP, followed by the 2 classes of disconnected graphs that fail.}
    \end{figure}
\end{example}

\begin{example}
    \label{Bad7verts}
    There are 79 simple connected graphs on 7 vertices that fail WLP. We show 8 of them in Figure \ref{fig:7-vertex-failures}. 
    \begin{figure}[h]
        \centering
        \includegraphics[width=5in]{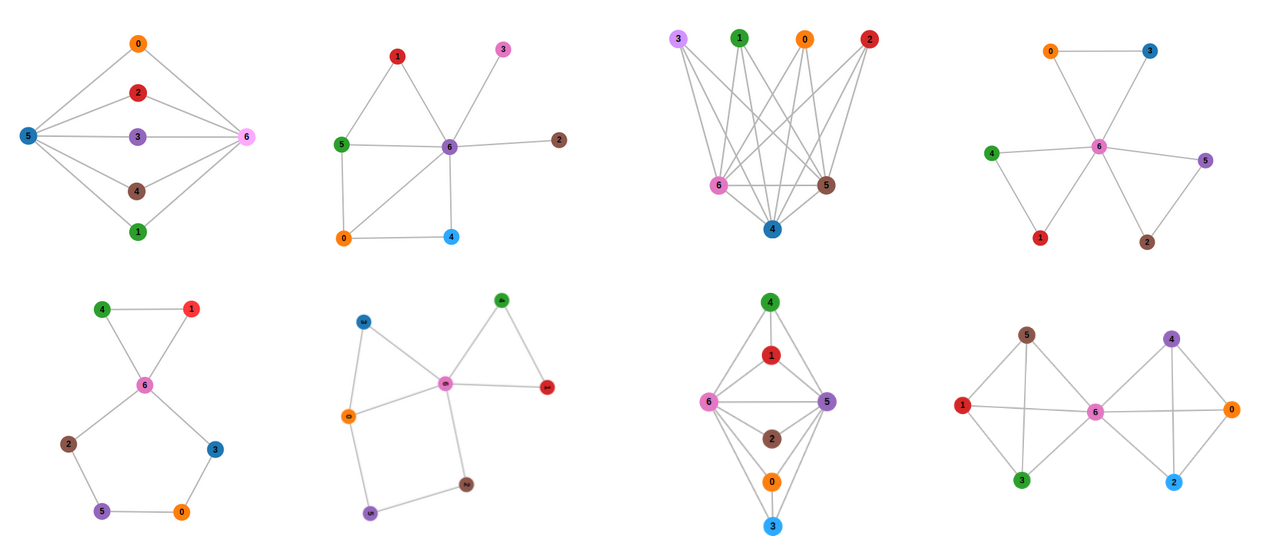}
        \caption{8 out of 79 isomorphism classes of simple, connected graphs on 7 vertices whose AOT algebras fail WLP.}
        \label{fig:7-vertex-failures}
    \end{figure}
    
    These are chosen to be a succinct, representative sample of the 79 in that most of the others have similar cycle structure to these graphs. For example, the second graph, consisting of three triangles glued together along two edges with the addition of two dangling edges, has 11 other graphs on 7 vertices with isomorphic AOT algebras due to different dangling edge placements. Likewise, the fourth graph, consisting of three triangles glued at a single vertex---or in other words the ``bowtie'' graph from Example \ref{bowtie} with another triangle glued at the center vertex---has AOT algebra isomorphic to that of the bowtie with a third triangle instead glued at any other vertex.

    As such, it turns out that while there are many isomorphism classes of graphs that fail WLP, there are a much smaller number of AOT algebra classes induced by graphic arrangements that fail.

\end{example}

\section{Graphic Artinian Orlik-Terao algebras with WLP}
A useful tool in proving that an algebra has WLP is provided by Theorem~\ref{Init} below: we could hope that for the AOT algebra of a graph that there is always a choice of term order such that the initial ideal also has WLP. Unfortunately, this is false, as we show in Proposition~\ref{BadSP}.

\begin{theorem}\label{Init}$[$Weibe, \cite{W}$]$
    If an initial ideal of $I$ has WLP, so does $I$.
\end{theorem}

Letting $I$ denote the OT ideal $I$ appearing in Definition \ref{OTdef}, and $J$ the ideal of the squares of the variables, we see that 
\[
    \IN(I+J) = J + M, \mbox{ where }M \mbox{ is a squarefree monomial ideal.}
\]

\begin{definition}
    Let $M$ be a matroid $M$ with an ordered ground set $[1,\ldots, n]$, and $C$ a circuit (minimal dependent set). Let $C=[a_1>a_2> \ldots .a_k]$ be a circuit. Then $[a_2, \ldots, a_k]$ is a {\em broken circuit}. For a graph $G$, the ground set consists of the edges, and chordless cycles are circuits. 
\end{definition}

In \cite{PS}, Proudfoot and Speyer show that the relations in Definition~\ref{OTdef} form a universal Gr\"obner basis for the Orlik-Terao ideal, and that for any term order $\prec$, the Stanley-Reisner ideal of the corresponding broken circuit complex is the initial ideal corresponding to the term order $\prec$.

We combine this with results of \cite{DN} and \cite{MNS} on monomial ideals containing the squares of variables. An important note is that Example~\ref{BadNews} shows for chordal graphs (which have a quadratic Gr\"obner basis), a term order which yields an initial ideal consisting of quadrics may not yield an initial ideal satisfying WLP, so that Theorem~\ref{Init} cannot be applied. Theorem 1.2 of \cite{sturmfels} shows that for a fixed ideal $I$, there are only finitely many possible initial ideals, and are encoded by the faces of the {\em state polytope}. The next proposition shows $I$ may have WLP, but no term order $\prec$ exists such that $\IN_{\prec}(I)$ has WLP. 

\begin{proposition}\label{BadSP}  
    There exist AOT algebras for which WLP holds, but does not hold for any initial ideal.
\end{proposition}

\begin{proof}
    Consider the $K_4$ graph, with a dangling edge depicted in Figure 5.
    \begin{figure}[h]
        \vskip -.10in
        \includegraphics[width=1.7in]{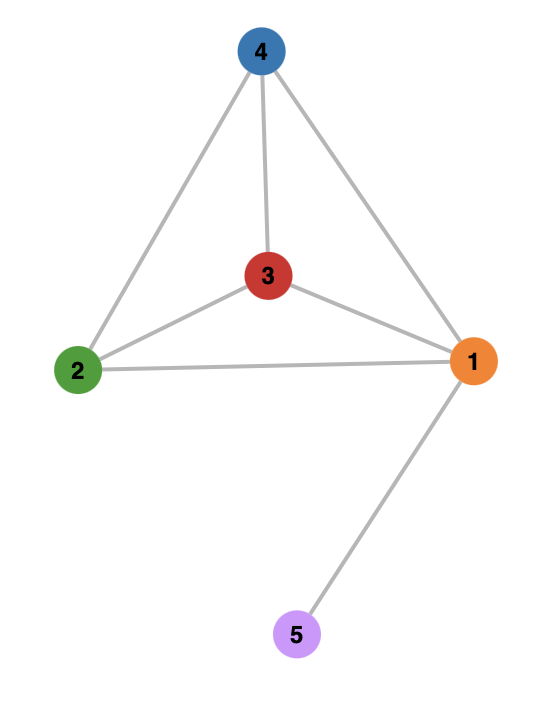} 
        \vskip -.17in
        \caption{AOT for $S/I$ has WLP, but no $S/\IN(I)$ has WLP}
    \end{figure}
    Using the {\tt Macaulay2} package {\tt Polyhedra} (see the appendix for a code snippet), we find there are 54 different possible initial ideals, and that none of them have WLP. Hence, Gr\"obner bases and degeneration to the initial ideal cannot, in general, certify WLP for the AOT ideal of a graphic arrangement.
\end{proof}

\subsection{Adding edges which introduce no new cycles}

Gr\"obner bases can be useful for some graph classes: we first recall a result of \cite{MNS}.

\begin{lemma}[Corollary 4.4, \cite{MNS}]
    \label{Tensor WLP iff l^2 full rank}
    For an Artinian $A$ which is a quotient by quadratic monomials with $A_{i+1} \neq 0$, $C = A \otimes \mathbb{K}[z]/z^{2}$ has WLP in degree $i$ if and only if $\cdot \ell^{2}$ has full rank on $A_{i-1}$, where $\ell$ is the sum of the variables of $A$.
\end{lemma}

\noindent The quadratic condition is not used in the proof in \cite{MNS}; the result holds without it.

\begin{proposition}\label{TreeCycle}
    Forests and cycles always have an AOT with WLP.
\end{proposition}

\begin{proof}
    For forests (i.e. acyclic graphs) the result is an immediate consequence of Stanley's theorem \cite{STAN} that WLP holds for monomial complete intersections, since for a forest, the AOT ideal consists only of the squares of the variables.

    For cycles, note that for any monomial ordering, the initial ideal of the Artinian Orlik-Terao algebra for the cycle graph on $n$ vertices is given by 
    \[
        \text{in}(I) = \langle y_{1}^{2}, y_{2}^{2}, \dots, y_{n}^{2}, y_{1} y_{2} \cdots y_{n-1} \rangle
    \]
    up to relabeling of the variables. 
    
    Let $A = \mathbb{K}[y_{1}, \dots, y_{n-1}] / \langle y_{1}^{2}, \dots, y_{n-1}^{2}, y_{1} y_{2} \cdots y_{n-1} \rangle$ and let $B = \mathbb{K}[y_{n}] / \langle y_{n}^{2} \rangle$. Hence,
    \begin{align*} 
        C &:= \mathbb{K}[y_{1},\dots,y_{n}] / \langle y_{1}^{2}, \dots, y_{n}^{2}, y_{1} y_{2} \cdots y_{n-1} \rangle \\ 
          &\cong A \otimes B.
    \end{align*}
    \noindent Thus, by Lemma \ref{Tensor WLP iff l^2 full rank}, we need only show that the multiplication map $\cdot \ell = \sum_{i=1}^{n-1} y_{i}$ for $A$ has the property that $\cdot \ell^{2}$ has full rank for all $i$.

    Note that $A$ is simply $D := \mathbb{K}[y_{1}, \dots, y_{n-1}] / \langle y_{1}^{2}, \dots, y_{n-1}^{2} \rangle$ but with the only monomial in degree $n-1$ thrown out: $A_{i} \cong D_{i}$ for $i=0, \dots, n-2$, $A_{n-1}=0$, $D_{n-1} \cong \mathbb{K}$, and $D_{n}=0$. Hence, in moving from $D$ to $A$, all multiplication maps of $D$ are preserved except for the last, which becomes zero. $D$ is a monomial complete intersection, and thus by Theorem 2.4 of \cite{STAN} $D$ has the Strong Lefschetz Property, so multiplication by $\cdot \ell^{d}$ has full rank for all $d$.
\end{proof}

There are other instances where the WLP for $\IN(I)$ allows us to certify WLP for the AOT algebra. For an algebra of the form
\[
    \Sym(V_1)/V_1^2 \bigotimes \Sym(V_2)/V_2^2 \bigotimes_{i=1}^n\KK[z_i]/z_i^2 
\] 
\noindent with both $\dim(V_1) \mbox{ and }\dim(V_2) \ge 2$, Proposition 4.7 in \cite{MNS} shows that WLP holds iff $n$ is odd. As a consequence, we have

\begin{example}\label{bowtieOdd}[Bowtie with dangling edges]
    The AOT algebra of the bowtie graph of Example~\ref{bowtie} has WLP when an odd number of edges are attached, as long as the attachment introduces no new cycles.
\end{example}

\noindent This follows from a more general result in \cite{MNS} which we present now:

\begin{lemma}[Theorem 4.8 \cite{MNS}]
    \label{Monomial tensor dimensions}
    If $char(\mathbb{K}) \neq 2$ and
    \[ C = \bigoplus_{i=1}^{n} Sym(V_{i})/V_{i}^{2} \text{ with } \text{dim}(V_{i}) = d_{i}, \]\
    then $C$ has WLP if and only if one of the following holds:
    \begin{enumerate}
        \item $d_{2}, \dots, d_{n} = 1$
        \item $d_{3}, \dots, d_{n} = 1$ and $n$ is odd.
    \end{enumerate}
\end{lemma}

For any number of vertices, this gives us a set of graphs whose AOT algebras have WLP. It also gives some graphs whose initial ideals fail WLP, but by Proposition \ref{BadSP}, this is not enough to say anything about the AOT algebras themselves.

\subsection{Using graphs that have WLP to build more}

The graphs shown to have $\text{in}(I)$ with WLP by Proposition \ref{TreeCycle} and Lemma \ref{Monomial tensor dimensions} will turn out to serve as ``atoms'' of sorts, from which we can build more graphs that satisfy WLP. Since $HS(AOT,t) = P({\mathcal A},t)$ and the rank of a graphic arrangement is the degree of the Poincar\'e polynomial, we have

\begin{lemma}
    \label{n vertices n HS}
    The Hilbert series of the AOT algebra of $G$ has degree $\kappa_0(G)-1$.
\end{lemma}

We present a theorem on graphs with 5 vertices. Example \ref{5-cycle and chord} shows the proof in action.

\begin{theorem}
    \label{5-vertex case}
    If $G$ is a graph on 5 vertices with a 5-cycle and an initial ideal $J$ that satisfies WLP, then adding a chord not already present in $G$ to the 5-cycle produces a graph $G'$ that also has WLP.
\end{theorem}

\begin{proof}
    Let $C$ be the AOT algebra of the extended graph $G'$ and $A$ be the tensor factor as in Lemma \ref{Tensor WLP iff l^2 full rank}. By Lemma \ref{n vertices n HS}, $C_{5} = 0$, and hence $A_{4} = 0$. Thus, to prove $WLP$ for $C$, if we use Lemma \ref{Tensor WLP iff l^2 full rank} we only need to show that the maps $\cdot \ell^{2}: A_{i-1} \to A_{i+1}$ for $i=1,2$, where $\cdot \ell$ is the sum of the variables of $A$, have full rank.

    The first map is trivially rank 1 and full, so let us consider $\cdot \ell^{2}: A_{1} \to A_{3}$. If we choose bases for $A_{1}$ and $A_{3}$, then up to scalar multiplication the matrix will be the set-inclusion matrix between the bases.

    We can label the edges of $G$ as $y_{i}$, $i = 1, \dots, e(G)$ so that the last 5 edges are those of the given 5-cycle, and the preceding variables are the chords of the cycle. Then if we use the broken circuit basis to construct our ideal with the ordering $y_{1} < \dots < y_{e(G)}$, the matrix for $\cdot \ell^{2}$ will be of the form
    \begin{small}
        \[
            \left[\!\begin{array}{cccc}
                   \vphantom{\left\{3,\:0\right\}} 1_{a_{1} \times 1} & * & * & * \\
                   \vphantom{\left\{3,\:0\right\}} 0 & 1_{a_{2} \times 1} & * & * \\
                     & \vdots &  &  \\
                   \vphantom{\left\{3,\:0\right\}} 0 & 0 & 1_{a_{m} \times 1} & * \\
                   \vphantom{\left\{3,\:0\right\}} 0 & 0 & 0 & U
                   \end{array}
            \right],
        \]
    \end{small}
    where each $1_{a_{i} \times 1}$ is a column vector of ones and
    \begin{small}
        \[
            U = \left[\!\begin{array}{cccc}
                   \vphantom{\left\{3,\:0\right\}} 1 & 1 & 1 & 0 \\
                   \vphantom{\left\{3,\:0\right\}} 1 & 1 & 0 & 1 \\
                   \vphantom{\left\{3,\:0\right\}} 1 & 0 & 1 & 1 \\
                   \vphantom{\left\{3,\:0\right\}} 0 & 1 & 1 & 1
                   \end{array}
            \right]
        \]
    \end{small}
    comes from the 5-cycle (see the below example). This matrix has full column rank, so $C$ has WLP.
\end{proof}

We now explicitly work out an example of the above phenomenon.

\begin{example}
    \label{5-cycle and chord}
    Let $G$ be the 5-cycle, and label its edges clockwise $y_{1}, \dots, y_{5}$. Proposition \ref{TreeCycle} tells us that $G$ has WLP, but let's verify and use the matrices as building blocks for other graphs. $G$ has exactly one dependency consisting of all the edges, and hence it has exactly one broken circuit. We can choose the ordering of the edges in this broken circuit basis so that the initial ideal of the AOT algebra $C$ for the graph $G$ is
    \[ C = \frac{\mathbb{K}[y_{1}, \dots, y_{5}]}{(y_{1}^{2}, \dots, y_{5}^{2}, y_{1}y_{2}y_{3}y_{4})} \cong \frac{\mathbb{K}[y_{1}, \dots, y_{4}]}{(y_{1}^{2}, \dots, y_{4}^{2}, y_{1}y_{2}y_{3}y_{4})} \otimes \frac{\mathbb{K}[y_{5}]}{(y_{5}^{2})}. \]
    Let $A$ be the first tensor factor on the right-hand side. Then $\HS(C, t) = 1 + 5t + 10t^{2} + 10t^{3} + 4t^{4}$ and $\HS(A, t) = 1 + 4t + 6t^{2} + 4t^{3}$.
    
    $C$ has WLP in degree 0 since the multiplication map $\cdot(y_{1}+y_{2}+y_{3}+y_{4}+y_{5})$ matrix (for suitable monomial bases) is a column of ones. $C$ also has WLP in degree 3 due to the above tensor decomposition. In fact, the first and last degrees for any graph have WLP.
    
    By Lemma \ref{Tensor WLP iff l^2 full rank}, $C$ has WLP in degrees 1 and 2 if and only if $\cdot (y_{1} + y_{2} + y_{3} + y_{4})^{2}$ has full rank on $A_{0}$ and $A_{1}$, respectively. We can choose a basis for each component of $A$ as follows: 
    \[
        \begin{array}{ccc}
        A_{0} & : & \{ 1 \} \\
        A_{1} & : & \{ y_{1}, y_{2}, y_{3}, y_{4} \} \\
        A_{2} & : & \{ y_{1}y_{2}, y_{1}y_{3}, y_{1}y_{4}, y_{2}y_{3}, y_{2}y_{4}, y_{3}y_{4} \} \\
        A_{3} & : & \{ y_{1}y_{2}y_{3}, y_{1}y_{2}y_{4}, y_{1}y_{3}y_{4}, y_{2}y_{3}y_{4} \}
        \end{array}
    \]
    Thus, the map $A_{0} \to A_{2}$ is simply a column of ones, while the map $A_{1} \to A_{3}$ is (up to multiplication by 2) the set-inclusion matrix
    \begin{small}
        \[
            \left[\!\begin{array}{cccc}
                   \vphantom{\left\{3,\:0\right\}} 1 & 1 & 1 & 0 \\
                   \vphantom{\left\{3,\:0\right\}} 1 & 1 & 0 & 1 \\
                   \vphantom{\left\{3,\:0\right\}} 1 & 0 & 1 & 1 \\
                   \vphantom{\left\{3,\:0\right\}} 0 & 1 & 1 & 1
                   \end{array}
            \right],
        \]
    \end{small}
    which has full rank.

    To see the previous theorem in action, let $G'$ be the 5-cycle with a chord, and label its edges as in $G$ with the chord edge labeled $y_{0}$. Then $G'$ has exactly three dependencies. We can choose the ordering of the edges in the broken circuit basis again so that the initial ideal of the AOT algebra $C'$ for the graph $G'$ is
    \begin{align*} 
        C' &= \frac{\mathbb{K}[y_{0}, \dots, y_{5}]}{(y_{0}^{2}, \dots, y_{5}^{2}, y_{0}y_{1}, y_{0}y_{3}y_{4}, y_{1}y_{2}y_{3}y_{4})} \\
           &\cong \frac{\mathbb{K}[y_{0}, \dots, y_{4}]}{(y_{0}^{2}, \dots, y_{4}^{2}, y_{0}y_{1}, y_{0}y_{3}y_{4}, y_{1}y_{2}y_{3}y_{4})} \otimes \frac{\mathbb{K}[y_{5}]}{(y_{5}^{2})}.
    \end{align*}
    Let $A'$ be the first tensor factor on the right-hand side. Then $\HS(C', t) = 1 + 6t + 14t^{2} + 15t^{3} + 6t^{4}$ and $\HS(A', t) = 1 + 5t + 9t^{2} + 6t^{3}$.

    Similarly to the case for $C$, $C'$ has WLP in degrees 0 and 3. We only need to check the square map for $A_{0}'$ (still a column of ones) and $A_{1}'$. We again choose a basis for each component of $A'$:
    \[
        \begin{array}{ccc}
        A_{0}' & : & \{ 1 \} \\
        A_{1}' & : & \{ y_{0}, y_{1}, y_{2}, y_{3}, y_{4} \} \\
        A_{2}' & : & \{ y_{0}y_{2}, y_{0}y_{3}, y_{0}y_{4}, y_{1}y_{2}, y_{1}y_{3}, y_{1}y_{4}, y_{2}y_{3}, y_{2}y_{4}, y_{3}y_{4} \} \\
        A_{3}' & : & \{ y_{0}y_{2}y_{3}, y_{0}y_{2}y_{4}, y_{1}y_{2}y_{3}, y_{1}y_{2}y_{4}, y_{1}y_{3}y_{4}, y_{2}y_{3}y_{4} \}
    \end{array}
    \]
    Note that indeed all basis elements for $A_{i}$ are preserved in passing to $A_{i}'$. The map $A_{1}' \to A_{3}'$ has matrix
    \begin{small}
        \[
            \left[\!\begin{array}{c|cccc}
                   \vphantom{\left\{3,\:0\right\}} 1 & 0 & 1 & 1 & 0 \\
                   \vphantom{\left\{3,\:0\right\}} 1 & 0 & 1 & 0 & 1 \\ \hline
                   \vphantom{\left\{3,\:0\right\}} 0 & 1 & 1 & 1 & 0 \\
                   \vphantom{\left\{3,\:0\right\}} 0 & 1 & 1 & 0 & 1 \\
                   \vphantom{\left\{3,\:0\right\}} 0 & 1 & 0 & 1 & 1 \\
                   \vphantom{\left\{3,\:0\right\}} 0 & 0 & 1 & 1 & 1
                   \end{array}
            \right],
        \]
    \end{small}
    The bottom-right $4 \times 4$ block comes directly from $A$, and the left-most column having a nonzero value and then only zeroes shows that this matrix has full rank if and only if the matrix from $A$ has full rank.
\end{example}

\subsection{Difficulties with 6-vertex graphs}

The above argument does not extend as easily to 6-vertex graphs due to an extra term in the Hilbert series from Lemma \ref{n vertices n HS}. We illustrate this with a similar example as above.

\begin{example}
    Consider the 6-cycle, and construct the AOT algebras $C$ and $A$ as in the 5-cycle case. Then $\HS(C, t) = 1 + 6t + 15t^{2} + 20t^{3} + 15t^{4} + 5t^{5}$ and $\HS(A, t) = 1 + 5t + 10t^{2} + 10t^{3} + 5t^{4}$. We care about the ranks of the maps $A_{0} \to A_{2}$ (always 1), $A_{1} \to A_{3}$, and, new here, $A_{2} \to A_{4}$. We are primarily interested in this final map, as the first two can be shown to be full rank using the above strategy. We can choose bases for the relevant components of $A$ as follows: 
    \[
        \begin{array}{ccc}
        A_{2} & : & \{ y_{1}y_{2}, y_{1}y_{3}, y_{1}y_{4}, y_{1}y_{5}, y_{2}y_{3}, y_{2}y_{4}, y_{2}y_{5}, y_{3}y_{4}, y_{3}y_{5}, y_{4}y_{5} \} \\
        A_{4} & : & \{ y_{1}y_{2}y_{3}y_{4}, y_{1}y_{2}y_{3}y_{5}, y_{1}y_{2}y_{4}y_{5}, y_{1}y_{3}y_{4}y_{5}, y_{2}y_{3}y_{4}y_{5}. \}
    \end{array}
    \]
    The map $\left( \sum y_{i} \right)^{2}$ from $A_{2} \to A_{4}$ produces the matrix
    \begin{small}
        \[
            \left[\!\begin{array}{cccccccccc}
                   \vphantom{\left\{3,\:0\right\}} 1 & 1 & 1 & 0 & 1 & 1 & 0 & 1 & 0 & 0 \\
                   \vphantom{\left\{3,\:0\right\}} 1 & 1 & 0 & 1 & 1 & 0 & 1 & 0 & 1 & 0 \\
                   \vphantom{\left\{3,\:0\right\}} 1 & 0 & 1 & 1 & 0 & 1 & 1 & 0 & 0 & 1 \\
                   \vphantom{\left\{3,\:0\right\}} 0 & 1 & 1 & 1 & 0 & 0 & 0 & 1 & 1 & 1 \\
                   \vphantom{\left\{3,\:0\right\}} 0 & 0 & 0 & 0 & 1 & 1 & 1 & 1 & 1 & 1
                   \end{array}
            \right],
        \]
    \end{small}
    As one would expect, this matrix is full rank; see the top left $4 \times 4$ submatrix and the bottom row of ones.

    However, let us examine what happens in the case where we add an edge $y_{0}$ between a pair of opposite vertices of the 6-cycle. Assume that the edges are ordered so that the dependencies are $y_{0}y_{1}y_{2}y_{3}$, $y_{0}y_{4}y_{5}y_{6}$, and $y_{1}y_{2}y_{3}y_{4}y_{5}y_{6}$. Then if $C'$ is the AOT algebra and $A'$ is the result of removing $y_{6}$ from the tensor decomposition, then with a similar choice of bases the map from $A_{2}' \to A_{4}'$ is
    \[
        \left[\!\begin{array}{ccccc|cccccccccc}
               \vphantom{\left\{3,\:0\right\}} 1 & 0 & 1 & 1 & 0 & 0 & 1 & 1 & 0 & 0 & 0 & 0 & 1 & 1 & 0 \\
               \vphantom{\left\{3,\:0\right\}} 1 & 0 & 1 & 0 & 1 & 0 & 1 & 0 & 1 & 0 & 0 & 0 & 0 & 0 & 0 \\
               \vphantom{\left\{3,\:0\right\}} 0 & 1 & 1 & 1 & 0 & 0 & 0 & 0 & 0 & 1 & 1 & 0 & 1 & 1 & 0 \\
               \vphantom{\left\{3,\:0\right\}} 0 & 1 & 1 & 0 & 1 & 0 & 0 & 0 & 0 & 1 & 0 & 1 & 0 & 0 & 0 \\ \hline
               \vphantom{\left\{3,\:0\right\}} 0 & 0 & 0 & 0 & 0 & 1 & 1 & 1 & 0 & 1 & 1 & 0 & 1 & 0 & 0 \\
               \vphantom{\left\{3,\:0\right\}} 0 & 0 & 0 & 0 & 0 & 1 & 1 & 0 & 1 & 1 & 0 & 1 & 0 & 1 & 0 \\
               \vphantom{\left\{3,\:0\right\}} 0 & 0 & 0 & 0 & 0 & 1 & 0 & 1 & 1 & 0 & 1 & 1 & 0 & 0 & 1 \\
               \vphantom{\left\{3,\:0\right\}} 0 & 0 & 0 & 0 & 0 & 0 & 1 & 1 & 1 & 0 & 0 & 0 & 1 & 1 & 1 \\
               \vphantom{\left\{3,\:0\right\}} 0 & 0 & 0 & 0 & 0 & 0 & 0 & 0 & 0 & 1 & 1 & 1 & 1 & 1 & 1
               \end{array}
        \right]
    \]
    The bottom right submatrix has rank 5 as above, but the top left submatrix only has rank 3. To show that the entire matrix has full rank, one option is to perform row and column operations to get the form
    \[
        \left[\!\begin{array}{ccc}
               \vphantom{\left\{3,\:0\right\}} U & V & W \\
               \vphantom{\left\{3,\:0\right\}} 0 & \text{I} & 0 \\
               \end{array}
        \right],
    \]
    and then show that the matrix $[U \quad W]$ has full row rank. This is difficult to generalize, though.
\end{example}

\section{Failure of WLP and tensor product decompositions}

In \cite{MMR}, Micha\l{}ek and Mir\'o-Roig prove results on WLP for quadratic algebras which have a tensor decomposition with only two factors of type $Sym(V_i)/V_i^2$, and \cite{MNS} builds on this, showing that quadratic algebras with more than two factors of type $Sym(V_i)/V_i^2$ always fail to have WLP.

\subsection{Tensor product decompositions}

A key ingredient in the analysis is a result of \cite{BMMRNZ}, which shows that for a tensor product $A = A' \otimes A''$ of two graded Artinian algebras, and corresponding linear forms $L' \in A'$ and $L'' \in A''$ with $L = L'+L''$, if $\mu_{L'}$ fails injectivity in degree $i-1$ and $\mu_{L''}$ fails injectivity in degree $j-1$, then $\mu_{L}$ fails injectivity in degree $i+j-1$, and similarly for surjectivity. In fact, we can be a bit more precise: in the setting of a tensor product algebra, there are tautological elements in the kernel of the multiplication map. 

In Example~\ref{bowtie} we displayed a tautological element of $\ker(\mu_L)$ with $L=\sum_{i=1}^6 a_iy_i$, arising from the fact that the algebras
\[
    \begin{array}{c}
        \KK[y_1,y_2,y_3]/\langle y_1^2,y_2^2,y_3^2,y_1y_2-y_1y_3+y_2y_3\rangle \\
        \mbox{ and }\\
        \KK[y_4,y_5,y_6]/\langle y_4^2,y_5^2,y_6^2,y_4y_5-y_4y_6+y_5y_6 \rangle.
    \end{array}
\]
both have Hilbert series $1+3t+2t^2$.

\begin{lemma}\label{taut}
    In the setting above, 
    \[
        \ker(\mu_{L'})_{i} \otimes \ker(\mu_{L''})_{j} \subseteq \ker(\mu_{L})_{i+j}.
    \]
    Denote the dimensions of the respective kernels as $k'_i$ and $k^{''}_j$. If $k'_i \cdot k^{''}_j > 0$, then $\mu_{L}$ fails both injectivity and surjectivity in degree $i+j$ if either of the following is true: 
    \begin{itemize}
        \item $\dim(A_{i+j}) \le \dim(A_{i+j+1})$.
        \item $\dim(A_{i+j}) - k'_i \cdot k^{''}_j < \dim(A_{i+j+1})$.  
    \end{itemize}
\end{lemma}

\begin{proof}
    For the first statement, let $\alpha \in \ker(\mu_{L'})_{i}$ and $\beta \in \ker(\mu_{L''})_{j}$. Then 
    \[
        \begin{array}{ccc} 
            \alpha \otimes \beta \otimes L &=& \alpha \otimes \beta \otimes L' +  \alpha \otimes \beta \otimes L^{''}\\
             & = & (\alpha \otimes L')\otimes \beta + \alpha \otimes(\beta \otimes L^{''})\\
             & = & 0+0
        \end{array}
    \]
    \noindent The second part of the lemma follows immediately.
\end{proof}

\begin{example}\label{manyTriangles}[Bouquet of triangles]
    Consider $d\ge 3$ triangles joined at a common vertex. Using Equation (1), the Hilbert series of the AOT algebra is
    $(1+t)^d(1+2t)^d$. If $(*) \dim(A_d) \le \dim(A_{d+1})$, then induction using a repeated application of Lemma~\ref{taut} shows that the map $A_d \longrightarrow A_{d+1}$ is not injective and therefore WLP fails; computations verify that $(*)$ holds for $d \le 30$.
\end{example}

\begin{example}\label{bowtieKer}
    For the bowtie graph, the kernel of $\mu_{\ell}: A_2\rightarrow A_3$ has dimension two, and in Example~\ref{bowtie} we identified one element of the kernel. For the other generator, we exploit the implicit bigrading in the tensor decomposition $A = A' \otimes A''$. Let
    \[
        \sum\limits_{i=1}^6 a_iy_i = \alpha + \beta,
    \]
    where $\alpha$ consists of the first three terms, and $\beta$ the second three terms, and let
    \[
        \begin{array}{ccc}
            \gamma & = & y_1y_3(a_1a_3+a_1a_2)+y_2y_3(a_2a_3-a_1a_2)\\
            \delta & = & y_4y_6(a_4a_6+a_4a_5)+y_5y_6(a_5a_6-a_4a_5)\\
            \epsilon& = & \alpha \cdot \beta = \sum\limits_{\substack{i=1..3\\j=4..6}}a_ia_jy_iy_j
        \end{array}
    \]
    A computation shows that 
    \[
        (2\gamma - \epsilon + 2\delta)\cdot (\alpha+\beta) = 0
    \]
    To see this directly, first note that since $\alpha^3 = 0 = \beta^3$, we have 
    
    \begin{equation}\label{cube0}
        (\alpha^2-\alpha \beta +\beta^2)(\alpha+\beta) = 0.
    \end{equation}
    
    Now recall that on $A'_2$ and $A''_2$ we have relations
    \[
        y_1y_2-y_1y_3+y_2y_3 = 0 = y_4y_5-y_4y_6+y_5y_6.
    \]
    Using these, we obtain
    \[
        \alpha^2=2(y_1y_3(a_1a_3+a_1a_2)+y_2y_3(a_2a_3-a_1a_2))=2 \gamma  \mbox{ in }A'_2 \subseteq A_2.
    \]
    \noindent and similarly for $\beta^2$. In particular, 
    \[
        \alpha^2 - \sum\limits_{\substack{i=1..3\\j=4..6}}a_ia_jy_iy_j + \beta^2
    \]
    \noindent is an element of the kernel of $\mu_{\ell}$. It is independent of the element we found using Lemma~\ref{taut} because the element from Lemma~\ref{taut} is of pure bidegree $(1,1)$, whereas the element above is of necessity of mixed bidegree.
\end{example}

\begin{proposition}\label{taut2}
    Let $A = A^{1} \otimes A^{2}$ with $A_{n}^{1} \neq 0 \neq A_{n}^{2}$ and $A_{n+1}^{1} = 0 = A_{n+1}^{2}$. Then $\ker(A_{n} \stackrel{\mu_{\ell}}{\rightarrow} A_{n+1})$ is nonzero, and if the kernels of the multiplication maps on $A_{i}^{1}$ and $A_{j}^{2}$ are nonzero for any pair of $i$ and $j$ such that $i + j = n$, then $\dim(\ker(A_{n} \stackrel{\mu_{\ell}}{\rightarrow} A_{n+1})) \geq 2$.
\end{proposition}

\begin{proof}
    Keeping the notation from Example~\ref{bowtieKer}, we have that a form of Equation~\ref{cube0} holds no matter what $n$ is as follows. If $n$ is odd,
    \[ (\alpha+\beta) \left( \alpha^{n-1} - \alpha^{n-2}\beta + \alpha^{n-3}\beta^{2} - \dots - \alpha\beta^{n-2} + \beta^{n-1} \right) = \alpha^{n} + \beta^{n} = 0. \]
    If $n$ is even,
    \[ (\alpha+\beta) \left( \alpha^{n-1} - \alpha^{n-2}\beta + \alpha^{n-3}\beta^{2} - \dots + \alpha\beta^{n-2} - \beta^{n-1} \right) = \alpha^{n} - \beta^{n} = 0. \]
    The degrees in either case are
    \[
        \begin{array}{cccc}
        \deg(\alpha^{n-1}) = (n-1,0), & \deg(\alpha^{n-2} \beta) = (n-2,1), & \dots & \deg(\beta^{n-1}) = (0,n-1)
        \end{array}
    \]
    There can be no cancellation between these terms due to the different bidegrees, and the first result follows; for the second apply Lemma~\ref{taut}, and note that due to bidegree considerations the relations from Lemma~\ref{taut} do not interact with the relation from Equation~\ref{cube0}.
\end{proof}

\section{Future Directions}

\begin{itemize}
\item In the case of the Orlik-Solomon algebra, the resonance varieties capture the locus of those linear forms $L_i$ where the multiplication map is not full rank. It would be interesting to connect this to work on the non-Lefschetz locus in \cite{NonLef}.
    \item Proposition~\ref{BadSP} shows that in general one cannot hope for a Gr\"obner basis characterization that covers WLP in all situations. So an interesting question is to identify classes of graphs for which there is a term order such that the corresponding Gr\"obner basis certifies WLP for the AOT.
    \item Example~\ref{bowtie} and Example~\ref{Bad6verts} suggest that there may be excluded minor conditions related to WLP. Find such conditions. 
    \item Further examine how WLP interacts with the graph theoretic deletion and contraction operations as in Theorem \ref{5-vertex case}, especially in the 6-vertex case and more.
    \item We computationally verified WLP up to 7 vertices, where the densest graphs' AOT algebras begin to have Hilbert series with terms in the thousands. It would be nice to build more examples, but the computational methods will need some more efficiencies before this is feasible.
\end{itemize}

\vskip .1in

\noindent{\bf Acknowledgements} Computations were performed using Macaulay2,
by Grayson and Stillman, and available at: {\tt http://www.math.uiuc.edu/Macaulay2/}; please see the appendix for relevant code snippets. 

We thank the organizers of the conference ``Artinian Algebras and Hilbert Schemes'' held at Universit\'e Cote d'Azur in June 2025 in honor of Tony Iarrobino for creating a wonderful collaborative environment, as well as two anonymous referees for helpful comments.

\bibliographystyle{amsplain}

\appendix

\section{Macaulay2 code}

\subsection{State polytope computations}

To compute the state polytope of the $K_{4}$ graph with a dangling edge from Proposition \ref{BadSP}, we use the following code.

\begin{verbatim}
-- Appendix A: Macaulay2 code verifying WLP failure
-- for all initial ideals of the AOT algebra of K_4
-- with a dangling edge.


loadPackage("HyperplaneArrangements", Reload=>true);
loadPackage("Polyhedra", Reload=>true);

-- Define the graph, K_4 + dangling edge, as a list.

G = {{1, 2}, {1, 3}, {1, 4}, {2, 3}, {2, 4}, {3, 4}, {4, 5}};

-- Compute the AOT ideal using HyperplaneArrangements.

-- Orlik-Terao ideal
OTideal = orlikTerao graphic G;
-- Squares of the variables
squares = ideal((gens ring OTideal) / (x -> x^2));
-- Artinian Orlik-Terao ideal
AOTideal = OTideal + squares

-- Compute the state polytope and all initial ideals
-- using Polyhedra.

-- For the first command, result is seq. of length 2: 
--   first elt is list of gens of all initial ideals,
--   second elt is state polytope.
initials = statePolytope AOTideal;
allInitialIdealGens = initials#0;
-- Number of distinct initial ideals; 54 here
#allInitialIdealGens
-- Turn gen sets into ideals
allInitialIdeals = allInitialIdealGens / (M -> ideal M);
-- Get all ambient rings
allAmbientRings = allInitialIdeals / (I -> ring I);

-- Define a function that determines whether an Artinian algebra
-- has WLP. Brute force.
-- Warning: uses a random linear form. 
-- If result is true, the algebra has WLP. 
-- If false, there is a very small probability that the algebra
-- truly has WLP.

WLP = (A) -> (
    -- Dimension check
    if dim(A)>0 then return "Error: algebra is not Artinian!";
    
    R := ambient A;
    -- Reproducibility; change this for different linear form
    setRandomSeed(1);
    -- Get just the linear form
    ell := (entries(random(R^{1}, R^1)))#0#0;
    i := 0;
    while basis(i,A) != 0 do (
        -- Constructs the matrix between bases
        M := ell * basis(i,A) // basis(i+1,A);
        -- Maximum possible rank
        greatestRank := min(numrows M, numcols M);
        if rank M < greatestRank then return false;
        i = i+1; -- Move on if full rank
    );
    return true;
)

-- Determine whether each initial ideal has WLP.

-- Apply the above function.
wlpValues = apply(
    allInitialIdeals, allAmbientRings, (I, S) -> (WLP(S/I))
);

-- Display the number of initial ideals that have WLP.
number(wlpValues, b -> b) -- 0 here
\end{verbatim}

\end{document}